\documentclass[12pt]{article}
\usepackage{amssymb}
\usepackage{amsthm}
\usepackage{amsmath}
\usepackage{color}
\usepackage[T1]{fontenc}

\newtheorem{definition}{Definition}
\newtheorem{lemma}[definition]{Lemma}
\newtheorem{proposition}[definition]{Proposition}
\newtheorem{theorem}[definition]{Theorem}
\newtheorem{remark}[definition]{Remark}
\newtheorem{example}[definition]{Example}

\title{Implication in weakly and dually weakly orthomodular lattices}
\author{Ivan~Chajda and Helmut~L\"anger}
\date{}
\begin{document}
\footnotetext[1]{Support of the research of both authors by \"OAD, project CZ~04/2017, and by IGA, project P\v rF~2018~012, as well as support of the second author by the Austrian Science Fund (FWF), project I~1923-N25, is gratefully acknowledged.}
\maketitle
\begin{abstract}
Weakly orthomodular and dually weakly orthomodular lattices were introduced by the authors in a recent paper. Similarly as for orthomodular lattices we try to introduce an implication in these lattices which can be easily axiomatized and which yields a nice lattice structure. As shown in the paper, this can be realized in several different ways. Moreover, we reveal the connection of weakly and dually weakly orthomodular lattices to residuated structures. Moreover, we provide a characterization of these lattices by means of certain generalized measures.
\end{abstract}



Orthomodular lattices were introduced independently by K.~Husimi (\cite H) and G.~Birkhoff and J.~von~Neumann (\cite{BvN}) as an algebraic semantics of the logic of quantum mechanics. The reader can find a detailed theory of orthomodular lattices in the monographs \cite B and \cite K. Implication in orthomodular lattices was investigated in \cite{CHL01} -- \cite{CHL08}. It was pointed out by the first author and S.~Bonzio in \cite{BC} that the orthomodular law may be satisfied also in complemented lattices which are not orthomodular because their complementation need not be an orthocomplementation. This fact was the starting point for the study \cite{CL18} by the authors where the concepts of a weakly orthomodular respectively dually weakly orthomodular lattice were introduced and elaborated. It was shown that these lattices have a number of features similar to those of orthomodular lattices despite the fact that de Morgan laws do not hold there. The classes of such lattices form varieties and it was shown that these varieties are congruence permutable and congruence regular and, of course, congruence distributive. It is a natural question if weakly orthomodular and/or dually weakly orthomodular lattices can serve as an algebraic semantics for a reasonable propositional logic. As usually, when a non-classical logic is investigated, one is focused on the logical connective implication because this connective is preferably used in logical deduction and derivations. For orthomodular lattices, there exist several ways how to introduce implication. The most popular is the so-called Sasaki implication but also other possible modifications were treated by several authors. Due to the lack of de Morgan laws for weakly and dually weakly orthomodular lattices, we have here more possibilities how to define an implication. Hence, our problem is to choose those which can be characterized within the theory of weakly and/or dually weakly orthomodular lattices by simple identities and, moreover, have a structure with a nice interpretation. This is the aim of the paper.

In the following all algebras considered are assumed to have non-empty universe.

We start with basic concepts. An {\em orthomodular lattice} is an algebra $\mathbf L=(L,\vee,\wedge,{}',0,1)$ of type $(2,2,1,0,0)$ such that $(L,\vee,\wedge,0,1)$ is a bounded lattice and $'$ is an {\em orthocomplementation}, i.e.\ it satisfies the identities $x\vee x'\approx1$, $x\wedge x'\approx0$, $(x\vee y)'\approx x'\wedge y'$ and $(x\wedge y)'\approx x'\vee y'$ and the {\em double negation law} $(x')'\approx x$, and moreover, it satisfies the so-called {\em orthomodular law}, i.e.
\[
x\leq y\text{ implies }y=x\vee(y\wedge x')
\]
or, equivalently,
\[
x\leq y\text{ implies }x=y\wedge(x\vee y').
\]
These conditions are equivalent to the following identities:
\begin{equation}\label{equ7}
(x\wedge y)\vee(x\wedge(x\wedge y)')\approx x,
\end{equation}
respectively
\begin{equation}\label{equ8}
(x\vee y)\wedge(x\vee(x\vee y)')\approx x
\end{equation}
which are used in the following definition.

\begin{definition}\label{def1}
{\rm(\cite{CL18})} A {\em lattice} $\mathbf L=(L,\vee,\wedge,{}')$ with a unary operation is called {\em weakly orthomodular} if it satisfies identity {\rm(\ref{equ7})} and {\em dually weakly orthomodular} if it satisfies identity {\rm(\ref{equ8})}. The algebra $\mathbf L$ is said to satisfy the {\em double negation law} if it satisfies the identity $(x')'\approx x$.
\end{definition}

It was shown by the authors in \cite{CL18} that conditions (\ref{equ7}) and (\ref{equ8}) are independent even for lattices with complementation. Of course, every orthomodular lattice is both weakly and dually weakly orthomodular. In a (dually) weakly orthomodular lattice the unary operation need neither be antitone nor an involution, see e.g.\ \cite{CL18} for examples. From Definition~\ref{def1} it is not clear if such a lattice is bounded and if the unary operation is a complementation and interchanges $0$ and $1$. This question is answered by the following lemma.

\begin{lemma}\label{lem3}
{\rm(\cite{CL18})} Every weakly orthomodular lattice $\mathbf L=(L,\vee,\wedge,{}')$ has a greatest element $1$ and satisfies the identity $x\vee x'\approx1$. If $\mathbf L$ has, additionally, a smallest element $0$ then $0'\approx1$. Every dually weakly orthomodular lattice $\mathbf L=(L,\vee,\wedge,{}')$ has a smallest element $0$ and satisfies the identity $x\wedge x'\approx0$. If $\mathbf L$ has, additionally, a greatest element $1$ then $1'\approx0$.
\end{lemma}

As shown in Lemma~\ref{lem3}, if a lattice $\mathbf L=(L,\vee,\wedge,')$ with a unary operation is both weakly and dually weakly orthomodular then $\mathbf L$ is bounded and $'$ a complementation on $\mathbf L$. However, $'$ need neither be antitone nor an involution and hence de Morgan laws need not hold in $\mathbf L$. Examples of both weakly and dually weakly orthomodular lattices whose unary operation is neither antitone nor an involution can be found in \cite{CL18}.

In Boolean algebras $\mathbf L=(L,\vee,\wedge,',0,1)$ the implication $x\rightarrow y$ is introduced as the term operation $x'\vee y$. This is the so-called ``classical implication''. In orthomodular lattices which are not Boolean algebras this kind of implication does not have the properties expected for implication and hence other versions of implication are studied, in particular the so-called {\em Sasaki implication} defined by
\[
x\rightarrow y:=(x'\wedge y')\vee y.
\]
(In an orthomodular lattice the {\em Sasaki projection} $p_b(a)$ of $a$ onto $b$ is defined by $(a\vee b')\wedge b$. Hence the Sasaki implication $x\rightarrow y$ in orthomodular lattices is nothing else than $(p_{y'}(x))'$.) This Sasaki implication was investigated by J.~C.~Abbott (\cite{A76}). Unfortunately, we cannot use it here in all the cases because of lack of de Morgan laws. First of all, we consider the implication defined by $x\rightarrow y:=(x\vee y)'\vee y$.

We call a binary operation $\rightarrow$ on a lattice $\mathbf L=(L,\vee,\wedge)$ a {\em {\rm D}-implication} on $\mathbf L$ if it satisfies the following identities:
\begin{align}
             (x\vee y)\rightarrow x & \approx y\rightarrow x,\label{equ1} \\
x\vee y\vee((x\vee z)\rightarrow y) & \approx z\rightarrow(x\vee y),\label{equ2} \\
    (x\vee y)\wedge(x\rightarrow y) & \approx y.\label{equ3}
\end{align}
(As we will see, {\rm D}-implications have something to do with dually weak orthomodularity.)

\begin{theorem}\label{th2}
Let $\mathbf L=(L,\vee,\wedge,0)$ be a lattice with $0$. Then the following hold:
\begin{enumerate}
\item[{\rm(i)}] If $\rightarrow$ is a {\rm D}-implication on $\mathbf L$ and $x':=x\rightarrow0$ for all $x\in L$ then $(L,\vee,\wedge,{}')$ is a dually weakly orthomodular lattice.
\item[{\rm(ii)}] If $(L,\vee,\wedge,{}')$ is a dually weakly orthomodular lattice and $x\rightarrow y:=(x\vee y)'\vee y$ for all $x,y\in L$ then $\rightarrow$ is a {\rm D}-implication on $\mathbf L$.
\item[{\rm(iii)}] The correspondence described in {\rm(i)} and {\rm(ii)} is one-to-one. Thus a binary operation $\rightarrow$ on $L$ is a {\rm D}-implication on $\mathbf L$ if and only if there exists a dually weakly orthomodular lattice $(L,\vee,\wedge,{}')$ such that $x\rightarrow y=(x\vee y)'\vee y$ for all $x,y\in L$.
\end{enumerate}
\end{theorem}

\begin{proof}
\
\begin{enumerate}
\item[(i)] Let $\rightarrow$ be a D-implication on $\mathbf L$ and $x':=x\rightarrow0$ for all $x\in L$. If $a,b\in L$ and $a\leq b$ then by applying (\ref{equ2}) and (\ref{equ3}) we compute
\begin{align*}
b\wedge(a\vee b') & =b\wedge(a\vee(b\rightarrow0))=b\wedge(a\vee0\vee((a\vee b)\rightarrow0))=b\wedge(b\rightarrow(a\vee0))= \\
& =(b\vee a)\wedge(b\rightarrow a)=a.
\end{align*}
Thus $(L,\vee,\wedge,{}')$ is a dually weakly orthomodular lattice.
\item[(ii)] Let $(L,\vee,\wedge,{}')$ be a dually weakly orthomodular lattice and $x\rightarrow y:=(x\vee y)'\vee y$ for all $x,y\in L$. Then
\begin{enumerate}
\item[(\ref{equ1})] $(x\vee y)\rightarrow x\approx(x\vee y\vee x)'\vee x\approx(y\vee x)'\vee x\approx y\rightarrow x$,
\item[(\ref{equ2})] $x\vee y\vee((x\vee z)\rightarrow y)\approx x\vee y\vee(x\vee z\vee y)'\vee y\approx(z\vee x\vee y)'\vee x\vee y\approx z\rightarrow(x\vee y)$,
\item[(\ref{equ3})] $(x\vee y)\wedge(x\rightarrow y)\approx(x\vee y)\wedge((x\vee y)'\vee y)\approx(x\vee y)\wedge(y\vee(x\vee y)')\approx y$.
\end{enumerate}
and hence $\rightarrow$ is a D-implication on $\mathbf L$.
\item[(iii)] If $\rightarrow$ is a D-implication on $\mathbf L$ and $x':=x\rightarrow0$ and $x\Rightarrow y:=(x\vee y)'\vee y$ for all $x,y\in L$ then
\[
x\Rightarrow y\approx(x\vee y)'\vee y\approx((x\vee y)\rightarrow0)\vee y\approx y\vee0\vee((y\vee x)\rightarrow0)\approx x\rightarrow(y\vee0)\approx x\rightarrow y.
\]
If, conversely, $(L,\vee,\wedge,{}')$ is a dually weakly orthomodular lattice and $x\rightarrow y:=(x\vee y)'\vee y$ and $\bar x:=x\rightarrow0$ for every $x,y\in L$ then
\[
\bar x\approx x\rightarrow0\approx(x\vee0)'\vee0\approx x'\vee0\approx x'.
\]
\end{enumerate}
\end{proof}

\begin{remark}
In Boolean algebras $(x\vee y)'\vee y$ coincides with the classical implication since
\[
(x\vee y)'\vee y\approx(x'\wedge y')\vee y\approx(x'\vee y)\wedge(y'\vee y)\approx(x'\vee y)\wedge1\approx x'\vee y.
\]
\end{remark}

The implication introduced before Theorem~\ref{th2} has a simple axiomatization and hence we are going to use it in our next investigations. However, we can also introduce another concept which enables us to assign every such implication a nice structure similarly as it was done for orthomodular lattices by Abbott (\cite{A76}), for Boolean algebras (\cite{A67}) or for MV-algebras (\cite{CHK}).

\begin{definition}
Let $\mathbf L=(L,\vee,\wedge)$ be a lattice. A {\em family of compatible dually weakly orthomodular sublattices of $\mathbf L$} is a family $(([x),\vee,\wedge,{}^x);x\in L)$ such that for every $x\in L$, $([x),\vee,\wedge,{}^x)$ is a dually weakly orthomodular sublattice of $\mathbf L$ and the {\em compatibility condition} $z^x\vee y=z^y$ for all $x,y,z\in L$ with $x\leq y\leq z$ holds. Here and in the following $[x)$ denotes the set $\{y\in L\mid x\leq y\}$.
\end{definition}

We are going to show that to every D-implication there can be assigned a family of compatible dually weakly orthomodular sublattices and that every D-implication can be constructed just by such a family.

\begin{theorem}\label{th1}
\
\begin{enumerate}
\item[{\rm(i)}] Let $\mathbf L=(L,\vee,\wedge)$ be a lattice and $\rightarrow$ a {\rm D}-implication on $\mathbf L$ and put $x^y:=x\rightarrow y$ for all $x,y\in L$. Then $(([x),\vee,\wedge,{}^x);x\in L)$ is a family of compatible dually weakly orthomodular sublattices of $\mathbf L$.
\item[{\rm(ii)}] Let $\mathbf L=(L,\vee,\wedge)$ be a lattice and $(([x),\vee,\wedge,{}^x);x\in L)$ a family of compatible dually weakly orthomodular sublattices of $\mathbf L$ and put $x\rightarrow y:=(x\vee y)^y$ for all $x,y\in L$. Then $\rightarrow$ is a {\rm D}-implication on $\mathbf L$.
\item[{\rm(iii)}] The correspondence between {\rm D}-implications on $\mathbf L$ and families of compatible dually weakly orthomodular sublattices of $\mathbf L$ described in {\rm(i)} and {\rm(ii)} is one-to-one.
\end{enumerate}
\end{theorem}

\begin{proof}
Let $a,b,c\in L$.
\begin{enumerate}
\item[(i)] At first we show how a D-implication induces a dually weakly orthomodular lattice $([a),\vee,\wedge,{}^a)$. Now, if $a\leq b$ then by applying identities (\ref{equ1}) -- (\ref{equ3}) we obtain
\[
a=(b\vee a)\wedge(b\rightarrow a)\leq b\rightarrow a=b^a.
\]
Moreover, if $a\leq b\leq c$ then
\begin{align*}
c\wedge(b\vee c^a) & =c\wedge(b\vee(c\rightarrow a))=c\wedge(b\vee a\vee((b\vee c)\rightarrow a))=c\wedge(c\rightarrow(b\vee a))= \\
                   & =(c\vee b)\wedge(c\rightarrow b)=b
\end{align*}								
and hence $([a),\vee,\wedge,{}^a)$ is a dually weakly orthomodular sublattice of $\mathbf L$. In order to check the compatibility condition we compute
\[
c^a\vee b=(c\rightarrow a)\vee b=b\vee a\vee((b\vee c)\rightarrow a)=c\rightarrow(b\vee a)=c\rightarrow b=c^b
\]
which proves the compatibility condition. Thus $(([x),\vee,\wedge,{}^x);x\in L)$ is a family of compatible dually weakly orthomodular sublattices of $\mathbf L$.
\item[(ii)] For the binary operation $\rightarrow$ on $L$ defined by $x\rightarrow y:=(x\vee y)^y$ for all $x,y\in L$ we compute
\begin{enumerate}
\item[(\ref{equ1})] $(x\vee y)\rightarrow x\approx(x\vee y\vee x)^x\approx(y\vee x)^x\approx y\rightarrow x$,
\item[(\ref{equ2})] $x\vee y\vee((x\vee z)\rightarrow y)\approx x\vee y\vee(x\vee z\vee y)^y\approx(x\vee y\vee z)^y\vee(x\vee y)\approx(x\vee y\vee z)^{x\vee y}\approx(z\vee x\vee y)^{x\vee y}=z\rightarrow(x\vee y)$,
\item[(\ref{equ3})] $(x\vee y)\wedge(x\rightarrow y)\approx(x\vee y)\wedge(x\vee y)^y\approx y$
\end{enumerate}
showing that $\rightarrow$ is a D-implication on $\mathbf L$.
\item[(iii)] If $\rightarrow$ is a D-implication on $\mathbf L$, $([x),\vee,\wedge,{}^x);x\in L)$ denotes the corresponding family of compatible dually weakly orthomodular sublattices of $\mathbf L$ and $\Rightarrow$ denotes the D-implication corresponding to this family then $\Rightarrow$ coincides with $\rightarrow$ since
\[
x\Rightarrow y\approx(x\vee y)^y\approx(x\vee y)\rightarrow y\approx(y\vee x)\rightarrow y\approx x\rightarrow y.
\]
Conversely, if $(([x),\vee,\wedge,{}^x);x\in L)$ is a family of compatible dually weakly orthomodular sublattices of $\mathbf L$, $\rightarrow$ denotes the corresponding D-implication and $(([x),\vee,$ $\wedge,{}_x);x\in L)$ denotes the family of compatible dually weakly orthomodular sublattices of $\mathbf L$ corresponding to this D-implication then the two families coincide since
\[
b_a=b\rightarrow a=(b\vee a)^a=b^a.
\]
\end{enumerate}
\end{proof}

Now, we turn to weakly orthomodular lattices. At first we investigate the Sasaki implication.

\begin{theorem}
Let $\mathbf L=(L,\vee,\wedge,\rightarrow,{}',0,1)$ be an algebra of type $(2,2,2,1,0,0)$ such that $(L,\vee,\wedge,0,1)$ is a bounded lattice and the identities
\begin{align*}
x\rightarrow y & \approx(x'\wedge y')\vee y, \\
         (x')' & \approx x
\end{align*}
are satisfied. Then $(L,\vee,\wedge,{}')$ is a weakly orthomodular lattice if and only if $\mathbf L$ satisfies the identities
\begin{align}
                                   0' & \approx1,\label{equ9} \\
(x\rightarrow0)\rightarrow(x\wedge y) & \approx x.\label{equ10}
\end{align}
\end{theorem}

\begin{proof}
First assume $(L,\vee,\wedge,{}')$ to be a weakly orthomodular lattice.
\begin{enumerate}
\item[(\ref{equ9})] This identity holds according to Lemma~\ref{lem3}.
\item[(\ref{equ10})] We have
\begin{align*}
(x\rightarrow0)\rightarrow(x\wedge y) & \approx(((x'\wedge0')\vee0)'\wedge(x\wedge y)')\vee(x\wedge y)\approx \\
                                      & \approx(((x'\wedge1)\vee0)'\wedge(x\wedge y)')\vee(x\wedge y)\approx \\
                                      & \approx((x'\vee0)'\wedge(x\wedge y)')\vee(x\wedge y)\approx((x')'\wedge(x\wedge y)')\vee(x\wedge y)\approx \\
																			& \approx(x\wedge(x\wedge y)')\vee(x\wedge y)\approx x.
\end{align*}
\end{enumerate}
Conversely, assume $\mathbf L$ to satisfy identities (\ref{equ9}) and (\ref{equ10}). Let $a,b\in L$ with $a\leq b$. Then
\begin{align*}
a\vee(b\wedge a') & =((b')'\wedge a')\vee a=((b'\vee0)'\wedge a')\vee a=(((b'\wedge1)\vee0)'\wedge a')\vee a= \\
                  & =(((b'\wedge0')\vee0)'\wedge a')\vee a=(b\rightarrow0)\rightarrow a=(b\rightarrow0)\rightarrow(b\wedge a)=b.
\end{align*}
Thus $(L,\vee,\wedge,{}')$ is a weakly orthomodular lattice.
\end{proof}

Let us note that in a Boolean algebra also the Sasaki implication $x\rightarrow y=(x'\wedge y')\vee y$ coincides with the classical implication $x'\vee y$.

For weakly orthomodular lattices we can define implication in a different way. The reason is that we would like to obtain a connective with nice properties and, moreover, it can be considered as a dual version of a D-implication in case the lattice is orthomodular.

We call a binary operation $\rightarrow$ on a lattice $\mathbf L=(L,\vee,\wedge,0)$ with $0$ a {\em {\rm W}-implication} on $\mathbf L$ if it satisfies the following identities:
\begin{align}
           (x\rightarrow0)\rightarrow0 & \approx x,\label{equ4} \\
((x\wedge y)\rightarrow0)\rightarrow x & \approx x,\label{equ5} \\
        (x\wedge y)\vee(x\rightarrow0) & \approx x\rightarrow y.\label{equ6}
\end{align}
(As we will see, {\rm W}-implications have something to do with weak orthomodularity.)

\begin{theorem}\label{th3}
Let $\mathbf L=(L,\vee,\wedge,0)$ be a lattice with $0$. Then the following hold:
\begin{enumerate}
\item[{\rm(i)}] If $\rightarrow$ is a {\rm W}-implication on $\mathbf L$ and $x':=x\rightarrow0$ for all $x\in L$ then $(L,\vee,\wedge,{}')$ is a weakly orthomodular lattice satisfying the double negation law.
\item[{\rm(ii)}] If $(L,\vee,\wedge,{}')$ is a weakly orthomodular lattice satisfying the double negation law and $x\rightarrow y:=x'\vee(x\wedge y)$ for all $x,y\in L$ then $\rightarrow$ is a {\rm W}-implication on $\mathbf L$.
\item[{\rm(iii)}] The correspondence described in {\rm(i)} and {\rm(ii)} is one-to-one. Thus a binary operation $\rightarrow$ on $L$ is a {\rm W}-implication on $\mathbf L$ if and only if there exists a weakly orthomodular lattice $(L,\vee,\wedge,{}')$ satisfying the double negation law such that $x\rightarrow y=x'\vee(x\wedge y)$ for all $x,y\in L$.
\end{enumerate}
\end{theorem}

\begin{proof}
\
\begin{enumerate}
\item[(i)] Let $\rightarrow$ be a W-implication on $\mathbf L$ and $x':=x\rightarrow0$ for all $x\in L$. If $a,b\in L$ and $a\leq b$ then by applying (\ref{equ4}) -- (\ref{equ6}) we compute
\[
a\vee(b\wedge a')=((a\rightarrow0)\wedge b)\vee((a\rightarrow0)\rightarrow0)=(a\rightarrow0)\rightarrow b=((b\wedge a)\rightarrow0)\rightarrow b=b.
\]
Thus $(L,\vee,\wedge,{}')$ is a weakly orthomodular lattice. According to (\ref{equ4}) it satisfies the double negation law.
\item[(ii)] Let $(L,\vee,\wedge,{}')$ be a weakly orthomodular lattice satisfying the double negation law and $x\rightarrow y:=x'\vee(x\wedge y)$ for all $x,y\in L$. Then $x\rightarrow0\approx x'\vee(x\wedge0)\approx x'\vee0\approx x'$. Thus
\begin{enumerate}
\item[(\ref{equ4})] $(x\rightarrow0)\rightarrow0\approx(x')'\approx x$,
\item[(\ref{equ5})] $((x\wedge y)\rightarrow0)\rightarrow x\approx(x\wedge y)'\rightarrow x\approx((x\wedge y)')'\vee((x\wedge y)'\wedge x)\approx(x\wedge y)\vee(x\wedge(x\wedge y)')\approx x$,
\item[(\ref{equ6})] $(x\wedge y)\vee(x\rightarrow0)\approx(x\wedge y)\vee x'\approx x'\vee(x\wedge y)\approx x\rightarrow y$
\end{enumerate}
and hence $\rightarrow$ is a W-implication on $\mathbf L$.
\item[(iii)] If $\rightarrow$ is a W-implication on $\mathbf L$ and $x':=x\rightarrow0$ and $x\Rightarrow y:=x'\vee(x\wedge y)$ for all $x,y\in L$ then
\[
x\Rightarrow y\approx x'\vee(x\wedge y)\approx(x\wedge y)\vee(x\rightarrow0)\approx x\rightarrow y.
\]
If, conversely, $(L,\vee,\wedge,{}')$ is a weakly orthomodular lattice satisfying the double negation law and $x\rightarrow y:=x'\vee(x\wedge y)$ and $\bar x:=x\rightarrow0$ for every $x,y\in L$ then
\[
\bar x\approx x\rightarrow0\approx x'\vee(x\wedge0)\approx x'\vee0\approx x'.
\]
\end{enumerate}
\end{proof}

\begin{remark}
In an orthomodular lattice $x'\vee(x\wedge y)=(p_x(y'))'$. Moreover, in Boolean algebras $x'\vee(x\wedge y)$ coincides with the classical implication since
\[
x'\vee(x\wedge y)\approx(x'\vee x)\wedge(x'\vee y)\approx1\wedge(x'\vee y)\approx x'\vee y.
\]
\end{remark}

Besides the previous results, we can list some elementary but important properties of the implications mentioned above.

\begin{proposition}
Let $\mathbf L=(L,\vee,\wedge,{}')$ be a dually weakly orthomodular lattice and $x\rightarrow y=(x\vee y)^y$. Then
\begin{enumerate}
\item[{\rm(i)}] $x\rightarrow0\approx x'$,
\item[{\rm(ii)}] if $\mathbf L$ has a $1$ then $x\rightarrow1\approx1$,
\item[{\rm(iii)}] if $\mathbf L$ has a $1$ then $1\rightarrow x\approx x$,
\item[{\rm(iv)}] if $\mathbf L$ has a $1$ then $x\rightarrow y=1$ implies $x\leq y$.
\end{enumerate}
\end{proposition}

\begin{proof}
\
\begin{enumerate}
\item[(i)] $x\rightarrow0\approx(x\vee0)'\vee0\approx x'$,
\item[(ii)] if $\mathbf L$ has a $1$ then $x\rightarrow1\approx(x\vee1)'\vee1\approx1$,
\item[(iii)] if $\mathbf L$ has a $1$ then $1\rightarrow x\approx(1\vee x)'\vee x\approx1'\vee x\approx0\vee x\approx x$,
\item[(iv)] if $\mathbf L$ has a $1$ then $x\rightarrow y=1$ implies $x\leq x\vee y=(x\vee y)\wedge1=(x\vee y)\wedge(x\rightarrow y)=(x\vee y)\wedge((x\vee y)'\vee y)=y$.
\end{enumerate}
\end{proof}

Similarly, we can list several interesting but elementary properties for the W-implication in weakly orthomodular lattices. According to Theorem~\ref{th3}, if $\mathbf L=(L,\vee,\wedge,{}')$ is a weakly orthomodular lattice then a W-implication can be expressed as
\[
x\rightarrow y:=x'\vee(x\wedge y)
\]
for all $x,y\in L$.

We can prove

\begin{proposition}
Let $\mathbf L=(L,\vee,\wedge,{}')$ be a weakly orthomodular lattice and define $x\rightarrow y:=x'\vee(x\wedge y)$ for all $x,y\in L$. Then
\begin{enumerate}
\item[{\rm(i)}] $x\rightarrow x\approx1$,
\item[{\rm(ii)}] $x\rightarrow1\approx1$,
\item[{\rm(iii)}] $x\vee(x\rightarrow y)\approx1$,
\item[{\rm(iv)}] $x\rightarrow(x\wedge y)\approx x\rightarrow y$,
\item[{\rm(v)}] $x\leq y$ implies $x\rightarrow y=1$,
\item[{\rm(vi)}] if $\mathbf L$ has a $0$ then $x\rightarrow0\approx x'$.
\end{enumerate}
\end{proposition}

\begin{proof}
\
\begin{enumerate}
\item[(i)] $x\rightarrow x\approx x'\vee(x\wedge x)\approx x'\vee x\approx1$,
\item[(ii)] $x\rightarrow1\approx x'\vee(x\wedge1)\approx x'\vee x\approx1$,
\item[(iii)] $x\vee(x\rightarrow y)\approx x\vee x'\vee(x\wedge y)\approx1$,
\item[(iv)] $x\rightarrow(x\wedge y)\approx x'\vee(x\wedge x\wedge y)\approx x'\vee(x\wedge y)\approx x\rightarrow y$,
\item[(v)] $x\leq y$ implies $x\rightarrow y=x'\vee(x\wedge y)=x'\vee x=1$,
\item[(vi)] if $\mathbf L$ has a $0$ then $x\rightarrow0\approx x'\vee(x\wedge0)\approx x'\vee0\approx x'$. 
\end{enumerate}
\end{proof}

Now, we want to investigate the question how to characterize weakly orthomodular or dually weakly orthomodular lattices via the systems of their principal filters or principal ideals, respectively, which are determined by special mappings of the lattice into the interval $[0,1]$ of the real numbers. Hence, such mappings can be considered as generalized measures.

\begin{definition}\label{def3}
For a lattice $\mathbf L=(L,\vee,\wedge,{}')$ with a unary operation we define
\begin{align*}
S_1(\mathbf L):=\{s\colon L\rightarrow[0,1]\mid s(x\vee(y\wedge x'))=s(y)\text{ for all }x,y\in L\text{ with }x\leq y\}, \\
S_2(\mathbf L):=\{s\colon L\rightarrow[0,1]\mid s(y\wedge(x\vee y'))=s(x)\text{ for all }x,y\in L\text{ with }x\leq y\}.
\end{align*}
\end{definition}

Let us mention that the sets $S_1(\mathbf L)$ and $S_2(\mathbf L)$ are non-empty since they contain every constant function from $L$ to $[0,1]$, in particular the constant functions with value $0$ respectively $1$.

\begin{theorem}
Let $\mathbf L=(L,\vee,\wedge,{}')$ be a lattice with a unary operation. Then the following hold:
\begin{enumerate}
\item[{\rm(i)}] $\mathbf L$ is weakly orthomodular if and only if for every $x\in L$ which is not the smallest element of $\mathbf L$ there exists some $s\in S_1(\mathbf L)$ satisfying $s^{-1}(\{1\})=[x)$.
\item[{\rm(ii)}] $\mathbf L$ is dually weakly orthomodular if and only if for every $x\in L$ which is not the greatest element of $\mathbf L$ there exists some $s\in S_2(\mathbf L)$ satisfying $s^{-1}(\{0\})=(x]$.
\end{enumerate}
\end{theorem}

\begin{proof}
Let $a,b\in L$ and assume $a\leq b$.
\begin{enumerate}
\item[(i)] If $\mathbf L$ is weakly orthomodular then for every $x\in L$ every $s\colon L\rightarrow[0,1]$ satisfying $s(y)=1$ if $y\geq x$ and $s(y)<1$ otherwise is an element of $S_1(\mathbf L)$ satisfying $s^{-1}(\{1\})=[x)$. Conversely, assume that for every $x\in L$ which is not the smallest element of $\mathbf L$ there exists some $s\in S_1(\mathbf L)$ satisfying $s^{-1}(\{1\})=[x)$. If $b$ is the smallest element of $\mathbf L$ then $a\vee(b\wedge a')=b$ since $a\vee(b\wedge a')\leq b$. Now assume $b$ not to be the smallest element of $\mathbf L$. Then there exists some $s\in S_1(\mathbf L)$ with $s^{-1}(\{1\})=[b)$. Now $s(a\vee(b\wedge a'))=s(b)=1$ and hence $a\vee(b\wedge a')\geq b$. Together with $a\vee(b\wedge a')\leq b$ we obtain $a\vee(b\wedge a')=b$.
\item[(ii)] If $\mathbf L$ is dually weakly orthomodular then for every $x\in L$ every $s\colon L\rightarrow[0,1]$ satisfying $s(y)=0$ if $y\leq x$ and $s(y)>0$ otherwise is an element of $S_2(\mathbf L)$ satisfying $s^{-1}(\{0\})=(x]$. Conversely, assume that for every $x\in L$ which is not the greatest element of $\mathbf L$ there exists some $s\in S_2(\mathbf L)$ satisfying $s^{-1}(\{0\})=(x]$. If $a$ is the greatest element of $\mathbf L$ then $b\wedge(a\vee b')=a$ since $b\wedge(a\vee b')\geq a$. Now assume $a$ not to be the greatest element of $\mathbf L$. Then there exists some $s\in S_2(\mathbf L)$ with $s^{-1}(\{0\})=(a]$. Now $s(b\wedge(a\vee b'))=s(a)=0$ and hence $b\wedge(a\vee b')\leq a$. Together with $b\wedge(a\vee b')\geq a$ we obtain $b\wedge(a\vee b')=a$.
\end{enumerate}
\end{proof}

One can ask why the function $s\colon L\rightarrow[0,1]$ was introduced in this way. The right answer is that it was mentioned as a generalized measure on a lattice with complementation. The problem is that we cannot assume de Morgan laws in our case. Thus relatively strong assumptions are necessary as stated in Definition~\ref{def3}. However, we are able to get a more detailed way how to introduce such functions.

For a lattice $(L,\vee,\wedge,{}')$ with a unary operation $'$ and a mapping $s\colon L\rightarrow[0,1]$ we define the following conditions:
\begin{enumerate}
\item[(i)] $s(1)=1$,
\item[(ii)] if $x,y\in L$ and $x\leq y$ then $s(x\vee y')=s(x)+s(y')$,
\item[(iii)] if $x,y\in L$ and $x\leq y$ then $s((x\vee y')\vee(x'\wedge y))=1$,
\item[(iv)] if $x,y\in L$ and $x\leq y$ then $s(x\vee(y\wedge(x\vee y'))')=1$.
\end{enumerate}
It should be mentioned that conditions (iii) and (iv) are not so strong as it looks at the first glance. Namely, if $(L,\vee,\wedge,{}')$ satisfies de Morgan laws then
\begin{align*}
(x\vee y')\vee(x'\wedge y) & \approx(x'\wedge y)'\vee(x'\wedge y)\approx1, \\
 x\vee(y\wedge(x\vee y'))' & \approx x\vee(y'\vee(x\vee y')')\approx(x\vee y')\vee(x\vee y')'\approx1.
\end{align*}
Thus (iii) and (iv) follow directly by the simple condition (i). In other words, this generalized measure $s$ does not distinguish certain terms that coincide whenever de Morgan laws hold.

Now we can state

\begin{proposition}
Let $\mathbf L=(L,\vee,\wedge,{}')$ be a lattice with a complementation and $s\colon L\rightarrow[0,1]$ satisfy {\rm(i)} and {\rm(ii)}. Then $s\in S_1(\mathbf L)$ if $s$ satisfies {\rm(iii)} and $s\in S_2(\mathbf L)$ if $s$ satisfies {\rm(iv)}.
\end{proposition}

\begin{proof}
Let $a,b\in L$. Since $a\leq a$ we have
\[
s(a)+s(a')=s(a\vee a')=s(1)=1
\]
according to (i) and (ii) whence $s(a')=1-s(a)$. Now assume $a\leq b$. Then $a\vee(b\wedge a')\leq b$ and
\[
s(a\vee(a'\wedge b))+s(b')=s((a\vee(a'\wedge b))\vee b')=s((a\vee b')\vee(a'\wedge b))=1
\]
according to (ii) and (iii) and hence
\[
s(a\vee(b\wedge a'))=1-s(b')=1-(1-s(b))=s(b),
\]
i.e.\ $s\in S_1(\mathbf L)$. Similarly, we have $a\leq b\wedge(a\vee b')$ and
\[
s(a)+s((b\wedge(a\vee b'))')=s(a\vee(b\wedge(a\vee b'))')=1
\]
according to (ii) and (iv) whence
\[
s(b\wedge(a\vee b')=1-(1-s(b\wedge(a\vee b')))=1-s((b\wedge(a\vee b'))')=s(a)
\]
proving $s\in S_2(\mathbf L)$.
\end{proof}

\begin{example}
Consider the lattice M$_3$ depicted in Fig.~1
\vspace*{-4mm}
\begin{center}
\setlength{\unitlength}{7mm}
\begin{picture}(6,6)
\put(3,1){\circle*{.3}}
\put(1,3){\circle*{.3}}
\put(3,3){\circle*{.3}}
\put(5,3){\circle*{.3}}
\put(3,5){\circle*{.3}}
\put(3,1){\line(-1,1)2}
\put(3,1){\line(0,1)4}
\put(3,1){\line(1,1)2}
\put(3,5){\line(-1,-1)2}
\put(3,5){\line(1,-1)2}
\put(2.85,.3){$0$}
\put(.4,2.85){$a$}
\put(3.3,2.85){$b$}
\put(5.3,2.85){$c$}
\put(2.85,5.35){$1$}
\put(2.2,-.4){{\rm Fig.~1}}
\end{picture}
\end{center}
\vspace*{-1mm}
where the unary operation $'$ is defined as follows:
\[
\begin{array}{c|ccccc}
 x & 0 & a & b & c & 1 \\
\hline
x' & 1 & b & c & a & 0
\end{array}
\]
Let $s\colon{\rm M}_3\rightarrow[0,1]$ be defined by
\[
\begin{array}{c|ccccc}
  x  & 0 &  a  &  b  &  c  & 1 \\
\hline	
s(x) & 0 & 1/2 & 1/2 & 1/2 & 1
\end{array}
\]
One can easily check that {\rm(i)} -- {\rm(iv)} are satisfied. Thus, $s\in S_1(\mathbf L)\cap S_2(\mathbf L)$.
\end{example}

A connection between orthomodular lattices and residuated structures was investigated by the authors in \cite{CL17a} and \cite{CL17b}. Despite the fact that the complementation in lattices being both weakly and dually weakly orthomodular need not be antitone, we are still able to find a residuated structure which can be assigned to every such lattice. This is important from the point of view of algebraic semantics because residuated structures are applied in algebraic semantics of several fuzzy logics. Thus, although we started with a generalization of the semantics of the logic of quantum mechanics, we can provide a bridge between these rather different approaches. For this reason, the W-implication introduced before Theorem~\ref{th3} turns out to be more appropriate.

\begin{lemma}\label{lem1}
Let $\mathbf L=(L,\vee,\wedge,{}')$ be a dually weakly orthomodular lattice and $a,b,c\in L$ and define
\begin{align*}
      x\odot y & :=(x\vee y')\wedge y, \\
x\rightarrow y & :=x'\vee(x\wedge y)
\end{align*}
for all $x,y\in L$. If $a\leq b\rightarrow c$ then $a\odot b\leq c$.
\end{lemma}

\begin{proof}
Using dual weak orthomodularity, if $a\leq b\rightarrow c$ we compute
\begin{align*}
a\odot b & =(a\vee b')\wedge b\leq((b\rightarrow c)\vee b')\wedge b=(b'\vee(b\wedge c)\vee b')\wedge b=b\wedge((b\wedge c)\vee b')= \\
         & =b\wedge c\leq c.
\end{align*}
\end{proof}

Consider the lattice ${\rm M}_3$ depicted in Fig.~1 where the unary operation $'$ is defined as follows:
\[
\begin{array}{c|ccccc}
 x & 0 & a & b & c & 1 \\
\hline
x' & 1 & b & c & b & 0
\end{array}
\]
It is easy to check that $(M_3,\vee,\wedge,{}')$ is both weakly and dually weakly orthomodular, but not orthomodular since $(a')'=b'=c\neq a$. Hence it does not satisfy the double negation law. On the other hand, $((x')')'\approx x'$. This is e.g.\ the condition for negation in intuitionistic logic. For such lattices we can state the following theorem.

\begin{theorem}\label{th5}
Let $\mathbf L=(L,\vee,\wedge,{}')$ be a weakly and dually weakly orthomodular lattice satisfying $((x')')'\approx x'$ and $a,b,c\in L$ and define
\begin{align*}
      x\odot y & :=(x\vee y')\wedge y, \\
x\rightarrow y & :=x'\vee(x\wedge y)
\end{align*}
for all $x,y\in L$. Then $\mathbf L$ satisfies the identity $(x\rightarrow y)\odot x\approx x\wedge y$. Moreover,
\[
a\odot(b')'\leq c\text{ if and only if }a\leq(b')'\rightarrow c.
\]
\end{theorem}

\begin{proof}
We have
\[
(x\rightarrow y)\odot x\approx(x'\vee(x\wedge y)\vee x')\wedge x\approx x\wedge((x\wedge y)\vee x')\approx x\wedge y.
\]
If $a\leq(b')'\rightarrow c$ then $a\odot(b')'\leq c$ according to Lemma~\ref{lem1}. Conversely, assume $a\odot(b')'\leq c$. Using weak orthomodularity and the identity $((x')')'\approx x'$ we compute
\begin{align*}
a & \leq a\vee b'=b'\vee((a\vee b')\wedge(b')')=b'\vee((b')'\wedge(a\vee((b')')')\wedge(b')')= \\
  & =b'\vee((b')'\wedge(a\odot(b')'))\leq b'\vee((b')'\wedge c)=((b')')'\vee((b')'\wedge c)=(b')'\rightarrow c.
\end{align*}
\end{proof}

For the reader's convenience, we repeat the following concept which is frequently used in the theory of residuated structures.

\begin{definition}\label{def2}
A {\em left residuated {\rm l}-groupoid} is an algebra ${\mathbf L}=(L,\vee,\wedge,\odot,\rightarrow,0,1)$ of type $(2,2,2,2,0,0)$ such that $(L,\vee,\wedge,0,1)$ is a bounded lattice, $x\odot1\approx1\odot x\approx x$ and for all $x,y,z\in L$, $x\odot y\leq z$ is equivalent to $x\leq y\rightarrow z$.
\end{definition}

For lattices being both weakly and dually weakly orthomodular and satisfying the double negation law we can prove a result which is stronger than that of Theorem~\ref{th5}, see the following theorem.

\begin{theorem}\label{th4}
Let $\mathbf L=(L,\vee,\wedge,{}')$ be a weakly and dually weakly orthomodular lattice satisfying the double negation law and $a,b\in L$ and define
\begin{align*}
      x\odot y & :=(x\vee y')\wedge y, \\
x\rightarrow y & :=x'\vee(x\wedge y)
\end{align*}
for all $x,y\in L$. Then $\mathbf L$ is bounded and $(L,\vee,\wedge,\odot,\rightarrow,0,1)$ a left residuated {\rm l}-groupoid satisfying the identity
\[
(x\rightarrow y)\odot x\approx x\wedge y.
\]
Moreover,
\begin{align*}
      a\odot b=a & \text{ if and only if }a\leq b, \\
a\rightarrow b=b & \text{ if and only if }a'\leq b.
\end{align*}
\end{theorem}

\begin{proof}
According to Lemma~\ref{lem3}, $\mathbf L$ is bounded and $1'\approx0$. Let $c\in L$. If $a\odot b\leq c$ then, due to weak orthomodularity of $\mathbf L$ and to the double negation law,
\begin{align*}
a & \leq a\vee b'=b'\vee((a\vee b')\wedge(b')')=b'\vee((a\vee b')\wedge b)=b'\vee(b\wedge(a\vee b')\wedge b)= \\
& =b'\vee(b\wedge(a\odot b))\leq b'\vee(b\wedge c)=b\rightarrow c.
\end{align*}
Conversely, if $a\leq b\rightarrow c$ then $a\odot b\leq c$ according to Lemma~\ref{lem1}. Moreover,
\begin{align*}
                x\odot1 & \approx(x\vee1')\wedge1\approx(x\vee0)\wedge1\approx x\wedge1\approx x, \\
               1\odot x & \approx(1\vee x')\wedge x\approx1\wedge x\approx x, \\
(x\rightarrow y)\odot x & \approx(x'\vee(x\wedge y)\vee x')\wedge x\approx x\wedge((x\wedge y)\vee x')\approx x\wedge y.
\end{align*}
Now the following are equivalent: $a\odot b=a$; $(a\vee b')\wedge b=a$; $b\wedge(a\vee b')=a$; $a\leq b$. Finally, the following are equivalent: $a\rightarrow b=b$; $a'\vee(a\wedge b)=b$; $a'\vee(b\wedge(a')')=b$; $a'\leq b$.
\end{proof}

Of course, every orthomodular lattice is both weakly and dually weakly orthomodular and satisfies the double negation law. In the following we present a lattice being both weakly and dually weakly orthomodular and satisfying the double negation law, but not being orthomodular.

\begin{example}
The lattice whose Hasse diagram is depicted in Fig.~2
\vspace*{-2mm}
\begin{center}
\setlength{\unitlength}{7mm}
\begin{picture}(12,8)
\put(3,1){\circle*{.3}}
\put(1,3){\circle*{.3}}
\put(3,3){\circle*{.3}}
\put(5,3){\circle*{.3}}
\put(9,3){\circle*{.3}}
\put(3,5){\circle*{.3}}
\put(7,5){\circle*{.3}}
\put(9,5){\circle*{.3}}
\put(11,5){\circle*{.3}}
\put(9,7){\circle*{.3}}
\put(3,1){\line(-1,1)2}
\put(3,1){\line(0,1)4}
\put(3,1){\line(1,1)2}
\put(3,1){\line(3,1)6}
\put(1,3){\line(1,1)2}
\put(1,3){\line(3,1)6}
\put(3,3){\line(3,1)6}
\put(5,3){\line(-1,1)2}
\put(5,3){\line(3,1)6}
\put(9,3){\line(-1,1)2}
\put(9,3){\line(0,1)4}
\put(9,3){\line(1,1)2}
\put(3,5){\line(3,1)6}
\put(7,5){\line(1,1)2}
\put(11,5){\line(-1,1)2}
\put(2.85,.3){$0$}
\put(.4,2.85){$a$}
\put(2.4,2.85){$b$}
\put(4.3,2.85){$c$}
\put(8.85,2.3){$d$}
\put(2.85,5.35){$e$}
\put(7.35,4.85){$f$}
\put(9.25,4.85){$g$}
\put(11.3,4.85){$h$}
\put(8.85,7.35){$1$}
\put(5.7,0){{\rm Fig.~2}}
\end{picture}
\end{center}
\vspace*{-1mm}
whose unary operation $'$ is defined by
\[
\begin{array}{c|cccccccccc}
x  & 0 & a & b & c & d & e & f & g & h & 1 \\
\hline
x' & 1 & g & h & f & e & d & c & a & b & 0
\end{array}
\]
is both weakly and dually weakly orthomodular and satisfies the double negation law, but it is not orthomodular since $'$ is not monotone since $a\leq f$, but $f'=c\not\leq g=a'$.
\end{example}

We can prove also the converse of Theorem~\ref{th4}.

\begin{theorem}
Let $\mathbf L=(L,\vee,\wedge,\odot,\rightarrow,0,1)$ be a left residuated {\rm l}-groupoid satisfying the identities
\begin{align*}
(x\rightarrow0)\rightarrow0 & \approx x, \\
                   x\odot y & \approx(x\vee(y\rightarrow0))\wedge y, \\
             x\rightarrow y & \approx(x\rightarrow0)\vee(x\wedge y)
\end{align*}
and define $x':=x\rightarrow0$ for all $x\in L$. Then $(L,\vee,\wedge,{}')$ is a weakly and dually weakly orthomodular lattice satisfying the double negation law.
\end{theorem}

\begin{proof}
Let $a,b\in L$ and $a\leq b$. Then $a\leq b'\vee(b\wedge a)=b\rightarrow a$ and hence $b\wedge(a\vee b')=a\odot b\leq a$ according to left-adjointness. This shows $b\wedge(a\vee b')\leq a$. Since the converse inequality is obvious, $(L,\vee,\wedge,{}')$ is dually weakly orthomodular. Now $a\vee(b\wedge a')\leq b$. To prove the converse inequality, we must use the double negation law. We have
\[
b\odot a'=(b\vee(a')')\wedge a=(b\vee a)\wedge a'=b\wedge a'\leq b.
\]
Thus, using left-adjointness, we obtain that $b\odot a'\leq b$ is equivalent to
\[
b\leq a'\rightarrow b=(a')'\vee(b\wedge a')=a\vee(b\wedge a')
\]
proving $a\vee(b\wedge a')\geq b$. Together we have $b=a\vee(b\wedge a')$. Thus $(L,\vee,\wedge,{}')$ is also weakly orthomodular and, obviously, satisfies the double negation law.
\end{proof}

Authors' addresses:

Ivan Chajda \\
Palack\'y University Olomouc \\
Faculty of Science \\
Department of Algebra and Geometry \\
17.\ listopadu 12 \\
771 46 Olomouc \\
Czech Republic \\
ivan.chajda@upol.cz

Helmut L\"anger \\
TU Wien \\
Faculty of Mathematics and Geoinformation \\
Institute of Discrete Mathematics and Geometry \\
Wiedner Hauptstra\ss e 8-10 \\
1040 Vienna \\
Austria, and \\
Palack\'y University Olomouc \\
Faculty of Science \\
Department of Algebra and Geometry \\
17.\ listopadu 12 \\
771 46 Olomouc \\
Czech Republic \\
helmut.laenger@tuwien.ac.at
\end{document}